\documentclass[11pt,a4paper]{article}

\usepackage[a4paper,margin = 2.5cm]{geometry}
\usepackage{amsmath,amsthm,amssymb}
\usepackage{mathrsfs}
\usepackage[utf8]{inputenc}
\usepackage{enumerate}
\usepackage{multicol}
\usepackage{graphicx}
\usepackage{color}
\usepackage{stmaryrd} % para los corchetes dobles: \llbracket, \rrbracket

\newtheorem{theorem}{Theorem}[section]
\newtheorem{lemma}[theorem]{Lemma}

\newtheorem{corollary}[theorem]{Corollary}
\newtheorem{example}[theorem]{Example}
\newtheorem{remark}[theorem]{Remark}

\newcommand{\imp}{\mathbin{\rightarrow}}

\begin{document}

\title{Strong standard completeness theorems for S5-modal \L ukasiewicz logics}

\author{D. Castaño, J. P. Díaz Varela, G. Savoy}

\maketitle

\begin{abstract}
We study the S5-modal expansion of the logic based on the \L ukasiewicz t-norm. We exhibit a finitary propositional calculus and show that it is finitely strongly complete with respect to this logic. This propositional calculus is then expanded with an infinitary rule to achieve strong completeness. These results are derived from properties of monadic MV-algebras: functional representations of simple and finitely subdirectly irreducible algebras, as well as the finite embeddability property. We also show similar completeness theorems for the extension of the logic based on models with bounded universe.
\end{abstract}

\section{Introduction}

In \cite{Hajek98} Hájek introduced an S5-modal expansion of any axiomatic extension $\mathcal{C}$ of his Basic Logic which is equivalent to the one-variable monadic fragment of the first-order extension $\mathcal{C}\forall$ of $\mathcal{C}$. We present next a slight generalization of his definition. Let $Prop$ be a countably infinite set of propositional variables, and let $Fm$ be the set of formulas built from $Prop$ in the language of Basic Logic expanded with two unary connectives $\square$ and $\lozenge$. Consider a class $\mathbb{C}$ of totally ordered BL-algebras. To interpret the formulas in $Fm$, consider triples $\mathbf{K} := \langle X,e,\mathbf{A} \rangle$ where $X$ is a non-empty set, $\mathbf{A} \in \mathbb{C}$, and $e\colon X \times Prop \to A$ is a function. The {\em truth value} $\|\varphi\|_{\mathbf{K},x}$ of a formula $\varphi$ in $\mathbf{K}$ at a point $x \in X$ is defined by recursion. For propositional variables $p \in Prop$ put $\|p\|_{\mathbf{K},x} := e(x,p)$. The definition of the truth value is then extended for the logical connectives in the language of Basic Logic in the usual way, and for the new unary connectives by
\[
\|\square \psi\|_{\mathbf{K},x} := \inf_{x' \in X} \|\psi\|_{\mathbf{K},x'}, \quad \text{ and } \quad \|\lozenge \psi\|_{\mathbf{K},x} := \sup_{x' \in X} \|\psi\|_{\mathbf{K},x'}.
\]
Note that the infima and suprema above may not exist in general in $\mathbf{A}$; hence, we restrict our attention to {\em safe} structures, that is, structures $\mathbf{K}$ for which $\|\varphi\|_{\mathbf{K},x}$ is defined for every $\varphi \in Fm$ at every point $x$. Given a $\Gamma \subseteq Fm$, we say that a safe structure $\mathbf{K}$ is a {\em model} of $\Gamma$ if $\|\varphi\|_{\mathbf{K},x} = 1$ for every $x \in X$ and $\varphi \in \Gamma$. For a set of formulas $\Gamma \cup \{\varphi\}$ we write $\Gamma \vDash_{\mathrm{S5}(\mathbb{C})} \varphi$ if every model of $\Gamma$ is also a model of $\varphi$. The logic thus defined depends on the class $\mathbb{C}$ and is denoted by $\mathrm{S5}(\mathbb{C})$. In case $\mathbb{C}$ is the class of totally ordered $\mathcal{C}$-algebras corresponding to an axiomatic extension $\mathcal{C}$ of Basic Logic we get the original definition given by Hájek; this logic was denoted by $\mathrm{S5}(\mathcal{C})$, but we reserve this notation for a related logic defined by means of an axiomatic system. 

In this article we are only interested in expansions of the infinite-valued \L ukasiewicz logic, which we denote by $\mathcal{L}$. Recall that the equivalent algebraic semantics of $\mathcal{L}$ is the variety $\mathbb{MV}$ of MV-algebras. We write $\mathbb{MV}_\mathrm{to}$ for the class of totally ordered MV-algebras. Thus, $\mathrm{S5}(\mathbb{MV}_\mathrm{to})$ is the S5-modal expansion of $\mathcal{L}$ defined by Hájek. Consider now the logic $\mathrm{S5}(\mathcal{L})$ on the same language as $\mathrm{S5}(\mathbb{MV}_\mathrm{to})$ defined by the following axiomatic system:
\begin{itemize} \setlength{\itemsep}{0pt}
\item Axioms:
\begin{itemize}
\item[] Instantiations of axiom-schemata of $\mathcal{L}$
\item[] $\square \varphi \imp \varphi$
\item[] $\varphi \imp \lozenge \varphi$
\item[] $\square(\nu \imp \varphi) \imp (\nu \imp \square \varphi)$
\item[] $\square(\varphi \imp \nu) \imp (\lozenge \varphi \imp \nu)$
\item[] $\square(\varphi \vee \nu) \imp (\square \varphi \vee \nu)$
\item[] $\lozenge (\varphi * \varphi) \equiv (\lozenge \varphi) * (\lozenge \varphi)$
\end{itemize}
where $\varphi$ is any formula, $\nu$ is any propositional combination of formulas beginning with $\square$ or $\lozenge$, and $\alpha \equiv \beta$ abbreviates $(\alpha \imp \beta) \wedge (\beta \imp \alpha)$.
\item Rules of inference:
\begin{itemize}
\item[] Modus Ponens: $\displaystyle \underline{\varphi, \varphi \imp \psi} \atop \displaystyle {\psi}$
\item[] Necessitation: $\displaystyle \underline{\phantom{..}\varphi\phantom{..}} \atop \displaystyle  {\square \varphi}$
\end{itemize}
\end{itemize}
In \cite{CCDVR21} the authors show a strong completeness theorem stating that $\mathrm{S5}(\mathbb{MV}_\mathrm{to}) = \mathrm{S5}(\mathcal{L})$. In this article we study the logic $\mathrm{S5}([0,1]_\textrm{\L})$ where $[0,1]_\textrm{\L}$ is the standard \L ukasiewicz t-norm on the unit real interval (of course, $\mathrm{S5}([0,1]_\textrm{\L})$ is a shorthand for $\mathrm{S5}(\{[0,1]_\textrm{\L}\})$). Note that $\mathrm{S5}([0,1]_\textrm{\L})$ is not finitary since it is a conservative expansion of the logic of $[0,1]_\textrm{\L}$, which is not finitary. Thus, a strong completeness theorem for $\mathrm{S5}(\mathcal{L})$ with respect to $\mathrm{S5}([0,1]_\textrm{\L})$ is not possible. We will show, however, that a finite strong completeness theorem does hold (Theorem \ref{TEO: finite strong completeness - general case}); in other words, $\mathrm{S5}(\mathcal{L})$ is the finitary companion of $\mathrm{S5}([0,1]_\textrm{\L})$. Weak completeness had already been proved by Rutledge in \cite{Rutledge59}. Our result generalizes Rutledge's and is obtained more directly.

In addition, we will show that adding one infinitary rule to the axiomatic system defining $\mathrm{S5}(\mathcal{L})$ is enough to obtain a logic $\mathrm{S5}(\mathcal{L})_\infty$ strongly complete with respect to $\mathrm{S5}([0,1]_\textrm{\L})$ (Theorem \ref{TEO: completitud estandar infinitaria}). This had already been shown for the propositional and first-order cases in \cite{Montagna07}. We follow the ideas in \cite{Kulacka18} and provide an adequate algebraic representation for simple algebras needed to obtain the monadic completeness theorem.

An interesting extension of $\mathrm{S5}(\mathbb{C})$ can be obtained by considering only safe structures $\mathbf{K} = (X,e,\mathbf{A})$ where $X$ has at most $k$ elements for a fixed positive integer $k$. We denote this logic by $\mathrm{S5}_k(\mathbb{C})$. We will show that an axiomatic extension of $\mathrm{S5}(\mathcal{L})$ by one axiom schema is finitely strongly complete with respect to $\mathrm{S5}_k([0,1]_\textrm{\L})$ (Theorem \ref{TEO: finite strong completeness - width k}). Moreover, the axiomatic extension of $\mathrm{S5}(\mathcal{L})_\infty$ obtained by adding the same axiom schema is strongly complete with respect to $\mathrm{S5}_k([0,1]_\textrm{\L})$ (Theorem \ref{TEO: strong standard completenes - width k}).

We use an algebraic method to prove the completeness results stated in the previous paragraphs. The representation theorems and properties that we prove here for monadic MV-algebras are also interesting in their own right since they improve our understanding of these structures.

We assume familiarity with the structural properties of MV-algebras. Most general facts about MV-algebras can be found in \cite{CDM00}. An important fact that we use is that the class of totally ordered MV-algebras has the amalgamation property. This is a direct consequence of the amalgamation property for totally ordered Abelian $\ell$-groups (see \cite{Peirce72}) by means of Mundici's functor. Another important result about totally ordered MV-algebras is that they enjoy the finite embeddability property. A proof in the context of Wajsberg hoops and a discussion of other proofs can be found in \cite{BF00}.

\section{Finitary S5-modal logics based on $[0,1]_\text{\rm \L}$}

\subsection{Algebraic results}

In this section we show most of the algebraic results needed to obtain the completeness theorems alluded to in the introduction. We start by reviewing the definition of monadic MV-algebras and their basic properties. Then we show a functional representation theorem for finitely subdirectly irreducible algebras and prove that this class of algebras has the finite embeddability property. Finally, we exhibit families of algebras that generate this class as a quasivariety; this leads directly to the main completeness result in the next section.

\subsubsection*{Basic definitions and results}

Monadic MV-algebras are MV-algebras endowed with two unary operations $\forall$ and $\exists$ that satisfy the following identities:
\begin{enumerate}[(M1)] \itemsep0em
\item\label{M1} $\forall x\imp x\approx 1$.
\item\label{M2} $\forall ( x\imp \forall y)\approx \exists x\imp \forall y$.
\item \label{M3}$\forall (\forall x\imp y)\approx \forall x\imp \forall y$.
\item\label{M4} $\forall ( \exists x\vee y)\approx \exists x\vee \forall y$.
\item\label{M5} $\exists (x*x)\approx \exists x*\exists x$.
\end{enumerate}
The previous definition is equivalent to the original one, given by Rutledge in \cite{Rutledge59}; see \cite{CDV14}. We denote by $\mathbb{MMV}$ the variety of monadic MV-algebras. This variety is the equivalent algebraic semantics of $\mathrm{S5}(\mathcal{L})$ (see \cite{CCDVR21}). We stick to the tradition of using $\forall$ and $\exists$ for monadic algebras; in contrast, we retain $\square$ and $\lozenge$ for logics. We usually write algebras in $\mathbb{MMV}$ as $\langle \mathbf{A}, \exists, \forall\rangle$, where we assume $\mathbf{A} \in \mathbb{MV}$.

\begin{example} \rm \label{EJ: canonical algebras}
Let $X$ be a non-empty set and let $[0,1]_\textrm{\rm \L}^X$ be the MV-algebra of functions from $X$ into the \L ukasiewicz t-norm $[0,1]_\textrm{\rm \L}$. Define $\exists_\vee, \forall_\wedge\colon [0,1]_\textrm{\rm \L}^X \to [0,1]_\textrm{\rm \L}^X$ in the following way
\[
\exists_\vee(f)(x) := \bigvee \{f(y): y \in X\} \qquad \text{ and }  \qquad \forall_\wedge(f)(x) := \bigwedge \{f(y): y \in X\}
\]
for $f \in [0,1]_\text{\rm \L}^X$ and $x \in X$. Note that $\exists_\vee(f)$ and $\forall_\wedge(f)$ are constant maps for each $f \in [0,1]_\text{\rm \L}^X$. The structure $\langle [0,1]_\text{\rm \L}^X, \exists_\vee, \forall_\wedge\rangle$ is a monadic MV-algebra. Note that, if $|X| = |Y|$, then $\langle [0,1]_\text{\rm \L}^X, \exists_\vee, \forall_\wedge\rangle$ is isomorphic to $\langle [0,1]_\text{\rm \L}^Y, \exists_\vee, \forall_\wedge\rangle$ through the obvious bijection. Also, if $|X| \leq |Y|$, there is an embedding from $\langle [0,1]_\text{\rm \L}^X, \exists_\vee, \forall_\wedge\rangle$ into $\langle [0,1]_\text{\rm \L}^Y, \exists_\vee, \forall_\wedge\rangle$. Indeed, without loss of generality assume $X \subseteq Y$ and fix $x_0 \in X$. Now map each $f\colon X \to [0,1]$ to the function $f'\colon Y \to [0,1]$ such that $f'(x) = f(x)$ for all $x \in X$ and $f'(y) = f(x_0)$ for all $y \in Y \setminus X$.

If $|X| = k$ is finite, we simply write $\langle [0,1]_\text{\rm \L}^k, \exists_\vee, \forall_\wedge\rangle$. If $X$ is infinite, $\langle [0,1]_\text{\rm \L}^X, \exists_\vee, \forall_\wedge\rangle$ generates $\mathbb{MMV}$ as a variety (see \cite{Rutledge59} or \cite{CCDVR21}). Moreover, we show below that the least quasi-variety that contains $\langle [0,1]_\text{\rm \L}^X, \exists_\vee, \forall_\wedge\rangle$ is also $\mathbb{MMV}$.

We can also define the monadic MV-algebra $\langle \mathbf{L}_m^n, \exists_\vee, \forall_\wedge\rangle$ where $\mathbf{L}_m$ is the $(m+1)$-element MV-chain and $\exists_\vee, \forall_\wedge$ are defined as above. Clearly $\langle \mathbf{L}_m^n, \exists_\vee, \forall_\wedge\rangle$ embeds into $\langle [0,1]_\text{\rm \L}^n, \exists_\vee,\forall_\wedge\rangle$.
\end{example}

The following lemma contains some structural properties of monadic MV-algebras. These properties are true for monadic BL-algebras (of which monadic MV-algebras are a special case); proofs can be found in \cite{CCDVR17}.

\begin{lemma} \label{LEMA: props basicas}
Let $\langle \mathbf{A}, \exists, \forall\rangle$ be a monadic MV-algebra. Then:
\begin{enumerate}[$(1)$]  \itemsep0em
\item $\exists A = \forall A$;
\item $\exists \mathbf{A}$ is a subalgebra of $\mathbf{A}$;
\item $\exists a = \min\{c \in \exists A: c \geq a\}$ and $\forall a = \max\{c \in \exists A: c \leq a\}$ for every $a \in A$;
\item the lattices of congruences of $\langle \mathbf{A}, \exists, \forall\rangle$ and $\exists \mathbf{A}$ are isomorphic;
\item $\langle \mathbf{A}, \exists, \forall\rangle$ is FSI if and only if $\exists \mathbf{A}$ is totally ordered;
\item $\langle \mathbf{A}, \exists, \forall\rangle$ is subdirectly irreducible if and only if $\exists \mathbf{A}$ is subdirectly irreducible;
\item $\langle \mathbf{A}, \exists, \forall\rangle$ is simple if and only if $\exists \mathbf{A}$ is simple.
\end{enumerate}
\end{lemma}

The next lemma includes several arithmetical properties that are used constantly throughout the paper. Again these properties are also true for monadic BL-algebras (see \cite{CCDVR17}).

\begin{lemma}
Let $\langle \mathbf{A}, \exists, \forall\rangle$ be a monadic MV-algebra. Then, for any $a,b \in A$ and $c \in \exists A$:
\begin{multicols}{2}
\begin{enumerate}[$(1)$]  \itemsep0em
\item $\forall 1 = \exists 1 = 1$ and $\forall 0 = \exists 0 = 0$;
\item $\forall c = \exists c = c$;
\item $\forall a \leq a \leq \exists a$;
\item if $a \leq b$, then $\forall a \leq \forall b$ and $\exists a \leq \exists b$;
\item $\forall (a \vee c) = \forall a \vee c$;
\item $\exists (a \vee b) = \exists a \vee \exists b$;
\item $\forall (a \wedge b) = \forall a \wedge \forall b$;
\item $\exists (a \wedge c) = \exists a \wedge c$;
\item $\forall (a \imp c) = \exists a \imp c$;
\item $\exists (a \imp c) \leq \forall a \imp c$;
\item $\forall (c \imp a) = c \imp \forall a$;
\item $\exists (c \imp a) \leq c \imp \exists a$;
\item $\forall \neg a = \neg \exists a$;
\item $\exists \neg a \leq \neg \forall a$.
\end{enumerate}
\end{multicols}
\end{lemma}

The following lemma collects some properties specific to monadic MV-algebras that are used below.

\begin{lemma} \label{LEMA: prop de los cuantif en MMV}
Let $\langle \mathbf{A}, \exists, \forall\rangle$ be a monadic MV-algebra. Then, for any $a \in A$:
\begin{enumerate}[$(1)$]  \itemsep0em
\item $\exists a = \neg \forall \neg a$ and $\forall a = \neg \exists \neg a$;
\item $\exists(\exists a \imp a) = 1$;
\item $\exists a^n = (\exists a)^n$ and $\exists na = n \exists a$;
\item $\forall a^n = (\forall a)^n$ and $\forall na = n \forall a$.
\end{enumerate}
\end{lemma}

\begin{proof}
Item (1) is a straightforward consequence of item (13) in the previous lemma and the involution of $\neg$. Item (2) follows from the fact that $\exists(\exists a \imp b) = \exists a \imp \exists b$ for any $a,b \in A$, which is proved in \cite{CCDVR17} to hold for every monadic MV-algebra. Finally, the proof of last two items can be found in \cite[Lemma 4.2]{CDV14}. 
\end{proof}

\subsubsection*{Functional representation}

An {\em $\mathcal{L}$-functional algebra} is an algebra $\langle \mathbf{A}, \exists, \forall\rangle$ such that there is an MV-chain $\mathbf{C}$ and a non-empty set $X$ with $\mathbf{A} \leq \mathbf{C}^X$ and the following holds:
\begin{equation} \label{EQ: testigos}
(\exists a)(x) = \bigvee \{a(y): y \in X\} \qquad \text{ and }  \qquad (\forall a)(x) = \bigwedge \{a(y): y \in X\} \tag{$*$}
\end{equation}
for all $a \in A$, $x \in X$. It is easy to check that $\mathcal{L}$-functional algebras are, in fact, FSI monadic MV-algebras. If $\mathbf{A}$ satisfies, in addition, that for each $a \in A$ there is $x_a \in X$ such that $(\forall a)(x) = a(x_a)$ for all $x \in X$, we say that $\mathbf{A}$ is an {\em $\mathcal{L}$-functional algebra with witnesses}. Observe that, since the identity $\exists x \approx \neg \forall \neg x$ holds in every monadic MV-algebra, the condition just stated is equivalent to the corresponding one for $\exists$.

\begin{theorem} \label{TEO: las fsi son funcionales con para todo que se alcanza}
Every FSI monadic MV-algebra is isomorphic to an $\mathcal{L}$-functional algebra with witnesses.
\end{theorem}

\begin{proof}
It follows from the proof of Theorem 3.8 in \cite{CCDVR21}.
\end{proof}

\begin{corollary} \label{CORO: las funcionales con para todo que se alcanza generan}
The class of $\mathcal{L}$-functional algebras with witnesses generates $\mathbb{MMV}$ as a quasi-variety.
\end{corollary}

\subsubsection*{Finite embedabbility property}

We now show that $\mathcal{L}$-functional algebras with witnesses can be partially embedded into finite algebras of the form $\langle \mathbf{L}_m^n, \exists_\vee, \forall_\wedge\rangle$ (see Example \ref{EJ: canonical algebras}).

\begin{lemma} \label{LEMA: fep para funcionales}
Let $\mathbf{B} := \langle \mathbf{A}, \exists, \forall\rangle$ be an $\mathcal{L}$-functional algebra with witnesses. Suppose $\mathbf{A} \leq \mathbf{C}^X$ for some MV-chain $\mathbf{C}$ and non-empty set $X$ such that for each $a \in A$ there is $x_a \in X$ such that $(\forall^\mathbf{B} a)(x) = a(x_a)$ for every $x \in X$. For any finite subset $S \subseteq A$ there are an algebra $\mathbf{D} := \langle \mathbf{L}_m^n, \exists_\vee, \forall_\wedge\rangle$, where $n$ and $m$ are natural numbers, $n \leq |X|$, and a one-to-one function $h\colon S \to L_m^n$ such that:
\begin{itemize}
\item $h(0^\mathbf{B}) = 0^\mathbf{D}$ if $0^\mathbf{B} \in S$,
\item $h(a_1 \imp^\mathbf{B} a_2) = h(a_1) \imp^\mathbf{D} h(a_2)$ if $a_1,a_2,a_1 \imp^\mathbf{B} a_2 \in S$,
\item $h(\forall^\mathbf{B} a) = \forall^\mathbf{D} h(a)$ if $a, \forall^\mathbf{B} a \in S$.
\end{itemize}
\end{lemma}

\begin{proof}
Let $X_0$ be a finite subset of $X$ such that:
\begin{itemize}
\item $x_{a} \in X_0$ for $a \in S$,
\item for $a,b \in S$ with $a \ne b$, there is $x \in X_0$ such that $a(x) \ne b(x)$.
\end{itemize}
Let $C_0 := \{a(x): a \in S, x \in X_0\} \cup \{0,1\}$. Since the class of MV-chains has the finite embeddability property, there is a finite MV-chain $\mathbf{L}_m$ and a one-to-one function $f\colon C_0 \to L_m$ such that
\begin{itemize}
\item $f(0^\mathbf{C}) = 0^{\mathbf{L}_m}$, $f(1^\mathbf{C}) = 1^{\mathbf{L}_m}$,
\item $f(c_1 \imp^\mathbf{C} c_2) = f(c_1) \imp^{\mathbf{L}_m} f(c_2)$ if $c_1, c_2, c_1 \imp^\mathbf{C} c_2 \in C_0$.
\end{itemize}
Consider the algebra $\mathbf{D} := \langle \mathbf{L}_m^{X_0}, \exists_\vee, \forall_\wedge\rangle$. Let $h\colon S \to L_m^{X_0}$ be defined by $h(a)(x) := f(a(x))$ for $a \in S$, $x \in X_0$. We prove now that $h$ has the desired properties.
\begin{itemize}
\item $h$ is one-to-one: Let $a,b \in S$ with $a \ne b$. There is $x \in X_0$ such that $a(x) \ne b(x)$. Since $f$ is one-to-one, $f(a(x)) \ne f(b(x))$, so  $h(a)(x) \ne h(b)(x)$. This shows that $h(a) \ne h(b)$.
\item Assume $0^\mathbf{B} \in S$. Then, for every $x \in X_0$ we have $h(0^\mathbf{B})(x) = f(0^\mathbf{B}(x)) = f(0^\mathbf{C}) = 0^{\mathbf{L}_m} = 0^\mathbf{D}(x)$. This shows that $h(0^\mathbf{B}) = 0^\mathbf{D}$.
\item Assume $a_1, a_2, a_1 \imp^\mathbf{B} a_2 \in S$. Then, for every $x \in X_0$ we have that
\begin{align*}
h(a_1 \imp^\mathbf{B} a_2)(x) & = f((a_1 \imp^\mathbf{B} a_2)(x)) \\
& = f(a_1(x) \imp^\mathbf{C} a_2(x)) \\
& = f(a_1(x)) \imp^{\mathbf{L}_m} f(a_2(x)) \\
& = h(a_1)(x) \imp^{\mathbf{L}_m} h(a_2)(x) \\
& = (h(a_1) \imp^\mathbf{D} h(a_2))(x).
\end{align*}
This shows that $h(a_1 \imp^\mathbf{B} a_2) = h(a_1) \imp^\mathbf{D} h(a_2)$.
\item Assume $a, \forall^\mathbf{B}a \in S$. By assumption, $(\forall^\mathbf{B}a)(x) = a(x_a)$ for every $x \in X$. In particular, for every $x \in X_0$ we have that $a(x_a) \imp^\mathbf{C} a(x) = 1^\mathbf{C}$, and, since $a(x_a), a(x), 1 \in C_0$, it follows that $f(a(x_a) \to^\mathbf{C} a(x)) = f(a(x_a)) \imp^{\mathbf{L}_m} f(a(x)) = f(1^\mathbf{C}) = 1^{\mathbf{L}_m}$. This proves that $f(a(x_a)) \leq^{\mathbf{L}_m} f(a(x))$ for every $x \in X_0$, that is, $h(a)(x_a) \leq^{\mathbf{L}_m} h(a)(x)$ for every $x \in X_0$. Thus, for every $x \in X_0$ we have that $(\forall^\mathbf{D} h(a))(x) = h(a)(x_a) = f(a(x_a)) = f((\forall^\mathbf{B}a)(x)) = h(\forall^\mathbf{B}a)(x)$. This shows that $\forall^\mathbf{D} h(a) = h(\forall^\mathbf{B}a)$.
\qedhere
\end{itemize}
\end{proof}

\begin{theorem}
$\mathbb{MMV}_\mathrm{FSI}$, and hence also $\mathbb{MMV}$, has the finite embeddability property. 
\end{theorem}

\begin{proof}
It follows immediately from Theorem \ref{TEO: las fsi son funcionales con para todo que se alcanza} and Lemma \ref{LEMA: fep para funcionales} that $\mathbb{MMV}_\textrm{FSI}$ has the finite embeddability property. An easy exercise proves then that $\mathbb{MMV}$ also has this property.
\end{proof}

\subsubsection*{Generating $\mathbb{MMV}$ as a quasivariety}

As another consequence of Lemma \ref{LEMA: fep para funcionales} we show several families of monadic MV-algebras that generate $\mathbb{MMV}$ as a quasi-variety; one of these families consists of only one (generic) algebra.

\begin{theorem} \label{TEO: MMV generada como cuasi por las finitas}
$\mathbb{MMV}$ is generated as a quasivariety by the algebras $\langle \mathbf{L}_m^n, \exists_\vee, \forall_\wedge\rangle$ for $m, n \geq 1$.
\end{theorem}

\begin{proof}
Let $$\varphi := \forall x_1,\ldots,x_n \, (s_1(\bar{x}) = t_1(\bar{x}) \mathbin{\&} \ldots \mathbin{\&} s_n(\bar{x}) = t_n(\bar{x}) \Rightarrow s(\bar{x}) = t(\bar{x}))$$ be a quasi-identity that is not valid in $\mathbb{MMV}$. We may assume, without loss of generality, that $\varphi$ is in the language $\{\to,0,\forall\}$, since all the other basic operations are expressible by terms in this sublanguage.

By Corollary \ref{CORO: las funcionales con para todo que se alcanza generan}, $\varphi$ does not hold in some $\mathcal{L}$-functional algebra $\mathbf{B} := \langle \mathbf{A}, \exists, \forall\rangle$ with witnesses. Thus, there are elements $a_1,\ldots,a_n \in A$ such that $s_i^\mathbf{B}(\bar{a}) = t_i^\mathbf{B}(\bar{a})$ for $1 \leq i \leq n$, but $s^\mathbf{B}(\bar{a}) \ne t^\mathbf{B}(\bar{a})$.

Let $S := \{r^\mathbf{B}(\bar{a}): r(\bar{x}) \text{ subterm of a term appearing in } \varphi\}$. By Lemma \ref{LEMA: fep para funcionales}, there are natural numbers $m,q$ and and a one-to-one map $f\colon S \to L_m^q$ such that:
\begin{itemize}
\item $h(0^\mathbf{B}) = 0^\mathbf{C}$ if $0^\mathbf{B} \in S$,
\item $h(a \imp^\mathbf{B} b) = h(a) \imp^\mathbf{C} h(b)$ if $a,b, a \imp^\mathbf{B} b \in S$,
\item $h(\forall^\mathbf{B} a) = \forall^\mathbf{C}h(a)$ if $a, \forall^\mathbf{B} a \in S$,
\end{itemize}
where $\mathbf{C} := \langle \mathbf{L}_m^q, \exists_\vee, \forall_\wedge\rangle$. A simple induction proves that for every subterm $r(\bar{x})$ of a term in $\varphi$, $h(r^\mathbf{B}(\bar{a})) = r^\mathbf{C}(h(\bar{a}))$, where $h(\bar{a})$ stands for $(h(a_1),\ldots,h(a_n))$. Thus, for $1 \leq i \leq r$ we have that $s_i^\mathbf{C}(h(\bar{a})) = h(s_i^\mathbf{B}(\bar{a})) = h(t_i^\mathbf{B}(\bar{a})) = t_i^\mathbf{C}(h(\bar{a}))$; however, since $s^\mathbf{B}(\bar{a}) \ne t^\mathbf{B}(\bar{a})$ and $h$ is one-to-one, we have that $h(s^\mathbf{B}(\bar{a})) \ne h(t^\mathbf{B}(\bar{a}))$, so $s^\mathbf{C}(h(\bar{a})) \ne t^\mathbf{C}(h(\bar{a}))$. This shows that $\varphi$ does not hold in $\mathbf{C}$.
\end{proof}

\begin{corollary} \label{COR: mmv generadas como cuasi por...}
\
\begin{enumerate}[$(1)$]
\item $\mathbb{MMV}$ is generated as a quasivariety by the algebras $\langle \mathbf{L}_m^\omega, \exists_\vee, \forall_\wedge\rangle$ for $m \geq 1$.
\item $\mathbb{MMV}$ is generated as a quasivariety by the algebras $\langle [0,1]_\text{\rm \L}^k, \exists_\vee, \forall_\wedge\rangle$ for $k \geq 1$.
\item $\mathbb{MMV}$ is generated as a quasivariety by the algebra $\langle [0,1]_\text{\rm \L}^\omega, \exists_\vee, \forall_\wedge\rangle$.
\end{enumerate}
\end{corollary}

\begin{proof}
The proof is immediate from the fact that $\langle \mathbf{L}_m^k, \exists_\vee, \forall_\wedge\rangle$ embeds into $\langle \mathbf{L}_m^\omega, \exists_\vee, \forall_\wedge\rangle$, into $\langle [0,1]_\text{\rm \L}^k, \exists_\vee, \forall_\wedge\rangle$ and into $\langle [0,1]_\text{\rm \L}^\omega, \exists_\vee, \forall_\wedge\rangle$ (see Example \ref{EJ: canonical algebras}).
\end{proof}

\subsection{Main completeness theorem}

In this section we derive the completeness theorem for the logic $\mathrm{S5}(\mathcal{L})$ from the algebraic results of the previous section.

Recall first the non-standard completeness theorem from \cite[Theorem 3.10]{CCDVR21} and the fact that $\mathbb{MMV}$ is the equivalent algebraic semantics of $\mathrm{S5}(\mathcal{L})$ from \cite{CCDVR17}.

\begin{theorem}[\cite{CCDVR21,CCDVR17}]
\
\begin{enumerate}[$(1)$]
\item $\mathrm{S5}(\mathcal{L})$ is strongly complete with respecto to $\mathrm{S5}(\mathbb{MV}_\mathrm{to})$.
\item $\mathbb{MMV}$ is the equivalent algebraic semantics of $\mathrm{S5}(\mathcal{L})$.
\end{enumerate}
\end{theorem}

We are now ready to prove one of the main results in the article.

\begin{theorem} \label{TEO: finite strong completeness - general case}
$\mathrm{S5}(\mathcal{L})$ is finitely strongly complete with respect to $\mathrm{S5}([0,1]_\textrm{\L})$.
\end{theorem}

\begin{proof}
Let $\varphi_1, \ldots, \varphi_n, \varphi$ be formulas in the monadic language. If $\{\varphi_1,\ldots,\varphi_n\} \vdash_{\mathrm{S5}(\mathcal{L})} \varphi$, then, by the (soundness part of the) non-standard completeness theorem (item 1 in the previous theorem), we have that $\{\varphi_1,\ldots,\varphi_n\} \vDash_{\mathrm{S5}(\mathbb{MV}_\mathrm{to})} \varphi$ and, in particular, $\{\varphi_1,\ldots,\varphi_n\} \vDash_{\mathrm{S5}([0,1]_\textrm{\L})} \varphi$.

Conversely, assume $\{\varphi_1,\ldots,\varphi_n\} \nvdash_{\mathrm{S5}(\mathcal{L})} \varphi$. Since $\mathbb{MMV}$ is the equivalent algebraic semantics of $\mathrm{S5}(\mathcal{L})$ (item 2 in the previous theorem), we have that $\mathbb{MMV} \nvDash \bigwedge_{i=1}^n \varphi_i \approx 1 \Rightarrow \varphi \approx 1$. Using Corollary \ref{COR: mmv generadas como cuasi por...}, there is a valuation $h$ in $\langle [0,1]_\textrm{\L}^\omega, \exists_\vee, \forall_\wedge\rangle$ such that $h(\varphi_i) = 1$ for $1 \leq i \leq n$, but $h(\varphi) \ne 1$. Let $\mathbf{K} := \langle \omega, e, [0,1]_\textrm{\L}\rangle$, where $e\colon Prop \times \omega \to [0,1]_\textrm{\L}$ is defined by $e(p,m) := h(p)(m)$ for $p \in Prop$ and $m \in \omega$. Then, $\mathbf{K}$ is a safe structure which is a model of $\{\varphi_1,\ldots,\varphi_n\}$, but not a model of $\varphi$. Thus, $\{\varphi_1,\ldots,\varphi_n\} \nvDash_{\mathrm{S5}([0,1]_\textrm{\L})} \varphi$.
\end{proof}

\subsection{The logic based on bounded universe models}

We can define an interesting extension of $\mathrm{S5}([0,1]_\textrm{\L})$ by restricting interpretations to models whose universes are of bounded size. More precisely, given a natural number $k$ we consider the logic $\mathrm{S5}_k([0,1]_\textrm{\L})$ defined in the same way as $\mathrm{S5}([0,1]_\textrm{\L})$ but restricting interpretations to models $\mathbf{K} := \langle X, e, \mathbf{A}\rangle$ where $|X| \leq k$. In this section we show that it suffices to add one axiom to $\mathrm{S5}(\mathcal{L})$ to get a formal system that is finitely strongly complete with respect to $\mathrm{S5}_k([0,1]_\textrm{\L})$. Moreover, we show the relation between these logics with monadic MV-algebras of width $\leq k$, already studied in \cite{CDV14}.

Let $\mathrm{S5}_k(\mathcal{L})$ be the axiomatic extension of $\mathrm{S5}(\mathcal{L})$ by the axiom schema
\begin{equation}
\bigwedge_{1 \leq i < j \leq k+1} \square(\varphi_i \vee \varphi_j) \imp \bigvee_{i=1}^{k+1} \square \varphi_i. \tag{$W_k$}
\end{equation}
Since $\mathbb{MMV}$ is the equivalent algebraic semantics of $\mathrm{S5}(\mathcal{L})$, it follows that the equivalent algebraic semantics of $\mathrm{S5}_k(\mathcal{L})$ is the subvariety of $\mathbb{MMV}$, which we denote by $\mathbb{MMV}_k$, determined by the equation:
\begin{equation}
\bigwedge_{1 \leq i < j \leq k+1} \forall(x_i \vee x_j) \imp \bigvee_{i=1}^{k+1} \forall x_i \approx 1. \tag{$W_k$}
\end{equation}
A monadic MV-algebra that satisfies identity $W_k$ is said to have {\em width} less than or equal to $k$. The following theorem explains this terminology. An {\em orthogonal set} in a (monadic) MV-algebra $\mathbf{A}$ is a subset $S \subseteq A \setminus \{1\}$ such that $x \vee y = 1$ for every $x,y \in S$, $x \ne y$.

\begin{theorem} \label{TEO: equivalencias para ancho k}
Let $\langle \mathbf{A}, \exists, \forall\rangle$ be an FSI monadic MV-algebra. The following conditions are equivalent:
\begin{enumerate}[$(1)$]
\item $\langle \mathbf{A}, \exists, \forall\rangle$ satisfies the identity $(W_k)$;
\item any orthogonal set in $\mathbf{A}$ has at most $k$ elements;
\item there are prime filters $P_1, \ldots, P_r$ in $\mathbf{A}$, $r \leq k$, such that $\bigcap_{i=1}^r P_i = \{1\}$ and $P_i \cap \exists A = \{1\}$ for $1 \leq i \leq r$.
\end{enumerate}
\end{theorem}

\begin{proof}
Identical to that in \cite[Theorem 2.1]{CCDVR21-Godel}.
\end{proof}

As a consequence of the last theorem it is clear that $\mathbb{MMV}_k \subseteq \mathbb{MMV}_{k+1}$ for every $k$. If a monadic MV-algebra satisfies $W_k$ but does not satisfy $W_{k-1}$ we say that it has {\em width $k$}. Observe that, by the last theorem, an FSI monadic MV-algebra of width $k$ must have $k$ prime filters $P_1, \ldots, P_k$ satisfying the conditions in item (3).

In order to obtain the necessary representation theorem for FSI algebras in $\mathbb{MMV}_k$ that yields the desired completeness theorem, we need first the following technical representation theorem.

\begin{lemma} \label{LEMA: ampliacion de filtros de Rutledge}
Let $\langle \mathbf{A}, \exists, \forall\rangle$ be an FSI monadic MV-algebra and let $\{P_i: i \in I\}$ be a family of prime filters of $\mathbf{A}$ such that $P_i \cap \exists A = \{1\}$ for each $i \in I$ and $\bigcap_{i \in I} P_i = \{1\}$. Then, there is a family of prime filters $\{Q_i: i \in I\}$ of $\mathbf{A}$ such that:
\begin{enumerate}[$(1)$]
\item $P_i \subseteq Q_i$ for each $i \in I$,
\item $Q_i \cap \exists A = \{1\}$ for each $i \in I$,
\item $\bigcap_{i \in I} Q_i = \{1\}$,
\item for every $i \in I$ and $a \in A \setminus Q_i$, there are a positive integer $n$ and an element $c \in \exists A \setminus \{1\}$ such that $a^n \imp c \in Q_i$.
\end{enumerate}
\end{lemma}

\begin{proof}
We first prove a property of the family $\{P_i: i \in I\}$: for every $a \in A$ there is $i \in I$ such that $a^2 \imp \forall a \in P_i$. Indeed, assume $a^2 \imp \forall a \not\in P_i$ for every  $i \in I$. Then $\forall a \imp a^2 \in P_i$ for every $i \in I$, so $\forall a \leq a^2$. Thus, $\forall a \leq \forall(a^2) = (\forall a)^2 \leq \forall a$, which shows that $(\forall a)^2 = \forall a$. Since $\exists \mathbf{A}$ is an MV-chain, it follows that $\forall a = 0$ or $\forall a = 1$. By assumption $\forall a \ne 1$, so $\forall a = 0$ and $2(\neg a) = \neg a^2 = a^2 \imp 0 \not\in P_i$ for every $i \in I$. But $2a \vee 2(\neg a) = 2(a \vee \neg a) = 1 \in P_i$, so $2a \in P_i$ for every $i \in I$. Hence $2a = 1$, and $1 = \forall 2a = 2 \forall a = 0$, a contradiction.

In order to define the filter $Q_i$ we first define a filter $F_i$ of the quotient algebra $\mathbf{A}/P_i$. For each $i \in I$ put
\[
F_i := \{a/P_i \in A/P_i: c/P_i < a^n/P_i \text{ for every } c \in \exists A \setminus \{1\} \text{ and every positive integer } n\}.
\]
We claim that $F_i$ is a filter of $\mathbf{A}/P_i$. It is easy to see that $1/P_i \in F_i$ and that $F_i$ is an increasing subset of $\mathbf{A}/P_i$. In addition, let $a_1/P_i,a_2/P_i \in F_i$, and consider a positive integer $n$ and an element $c \in \exists A \setminus \{1\}$. Note that, since $P_i$ is a prime filter, $\mathbf{A}/P_i$ is totally ordered, so we can assume $a_1/P_i \leq a_2/P_i$. Hence, $(a_1a_2)^n/P_i \geq (a_1 \wedge a_2)^{2n}/P_i = a_1^{2n}/P_i > c/P_i$. Hence $a_1a_2/P_i \in F_i$.

Now define $Q_i := \{a \in A: a/P_i \in F_i\}$ for $i \in I$. Clearly $Q_i$ is a prime filter on $\mathbf{A}$ containing $P_i$. It remains to show that the family $\{Q_i: i \in I\}$ has the properties 2-4 stated in the Lemma. To show property 2, suppose $Q_i \cap \exists A \ne \{1\}$. Let $c \in \exists A \setminus \{1\}$ be such that $c \in Q_i$. Then $c/P_i \in F_i$, which is a contradiction from the definition of $F_i$ (take $n = 1$). Now we prove property 3. Let $a \in \bigcap_{i \in I} Q_i$. Using the property of the family $\{P_i: i \in I\}$ shown at the start of the proof, there is $i \in I$ such that $a^2 \imp\forall a \in P_i$, so $a^2 \imp \forall a \in Q_i$. Since $a \in Q_i$, it follows that $\forall a \in Q_i$. Thus, $\forall a = 1$, so $a = 1$. Finally we prove property 4. Let $a \in A \setminus Q_i$. Then there is $c \in \exists A \setminus \{1\}$ and a positive integer $n$ such that $a^n/P_i \leq c/P_i$, that is, $a^n \imp c \in P_i \subseteq Q_i$.
\end{proof}

Observe that, if we apply Theorem \ref{TEO: las fsi son funcionales con para todo que se alcanza} to an FSI monadic MV-algebra $\langle \mathbf{A}, \exists, \forall\rangle$ of width $k$, we obtain an MV-chain $\mathbf{C}$ and a non-empty set $X$ such that $\mathbf{A} \leq \mathbf{C}^X$ and the equations in \eqref{EQ: testigos} (see p. \pageref{EQ: testigos}) are satisfied. However, Theorem \ref{TEO: las fsi son funcionales con para todo que se alcanza} says nothing about $|X|$. The following theorem guarantees that we can take $|X| = k$ for algebras of width $k$.

\begin{theorem} \label{TEO: FSI de ancho k son k funcionales}
Let $\langle \mathbf{A}, \exists, \forall\rangle$ be an FSI monadic MV-algebra of width $k$. Then there is a totally ordered MV-algebra $\mathbf{V}$ and an embedding $\alpha\colon \langle \mathbf{A},\exists, \forall\rangle \to \langle \mathbf{V}^k, \exists_\vee, \forall_\wedge\rangle$.
\end{theorem}

\begin{proof}
Since $\langle \mathbf{A},\exists,\forall\rangle$ is an FSI monadic MV-algebra of width $k$, there are prime filters $P_1,\ldots,P_k$ such that $\bigcap P_i = \{1\}$ and $P_i \cap \exists A = \{1\}$ for $i \in \{1,\ldots,k\}$. By Lemma \ref{LEMA: ampliacion de filtros de Rutledge} there are prime filters $Q_1, \ldots, Q_k$ with properties 1-4 stated in the Lemma. We claim that for every $a \in A$ there is $i \in \{1,\ldots,k\}$ such that $\exists a \imp a \in Q_i$. Indeed, suppose $\exists a \imp a \notin Q_i$ for any $i \in \{1,\ldots,k\}$. Using property $4$, there are elements $c_1,\ldots,c_k \in \exists A \setminus \{1\}$ and positive integers $n_1,\ldots,n_k$ such that $(\exists a \imp a)^{n_i} \imp c_i \in Q_i$ for $i \in \{1,\ldots,k\}$. Letting $n := \max \{n_1,\ldots,n_k\}$ and $c := \max\{c_1,\ldots,c_k\}$, we get that  $(\exists a \imp a)^n \imp c \in \bigcap Q_i = \{1\}$. Thus $(\exists a \imp a)^n \leq c$. Hence $c \geq \exists((\exists a \imp a)^n) = (\exists(\exists a \imp a))^n = 1$, a contradiction.

Now consider the embeddings $\nu_i|_{\exists A}\colon \exists \mathbf{A} \to \mathbf{A}/Q_i$ for $i \in \{1,\ldots,k\}$, obtained by restriction of the canonical maps $\nu_i\colon \mathbf{A} \to \mathbf{A}/Q_i$. Since the class of totally ordered MV-algebras has the amalgamation property, there is a totally ordered MV-algebra $\mathbf{V}$ and embeddings $\beta_i\colon \mathbf{A}/Q_i \to \mathbf{V}$ such that $\beta_i \circ \nu_i|_{\exists a} = \beta_j \circ \nu_j|_{\exists a}$ for $i,j \in \{1,\ldots,k\}$. Let $\alpha\colon \mathbf{A} \to \mathbf{V}^k$ be the embedding given by $\alpha(x) := (\beta_1(x/Q_1),\ldots,\beta_k(x/Q_k))$. We claim that $\alpha(\exists a) = \exists_\vee \alpha(a)$ for every $a \in A$. Indeed, for $a \in A$ there is $i \in \{1,\ldots,k\}$ such that $\exists a \imp a \in Q_i$, so $a/Q_i = \exists a/Q_i$, which implies that $\beta_j(a/Q_j) \leq \beta_j(\exists a/Q_j) = \beta_i(\exists a/Q_i) = \beta_i(a/Q_i)$ for any $j \in \{1,\ldots,k\}$.
\end{proof}

\begin{theorem} 
$\mathbb{MMV}_k$ is generated as a quasivariety by the algebras $\langle \mathbf{L}_m^k, \exists_\vee, \forall_\wedge\rangle$ for $m \geq 1$.
\end{theorem}

\begin{proof}
The proof is the same as that of Theorem \ref{TEO: MMV generada como cuasi por las finitas} using Theorem \ref{TEO: FSI de ancho k son k funcionales} and Lemma \ref{LEMA: fep para funcionales}.
\end{proof}

\begin{corollary} \label{COR: mmv de ancho k generadas por 01 a la k}
$\mathbb{MMV}_k$ is generated as a quasivariety by the algebra $\langle [0,1]_\text{\rm \L}^k, \exists_\vee, \forall_\wedge\rangle$.
\end{corollary}

We are now ready to prove that $\mathrm{S5}_k(\mathcal{L})$ is finitely strongly complete with respect to the logic $\mathrm{S5}_k([0,1]_\textrm{\L})$.

\begin{theorem} \label{TEO: finite strong completeness - width k}
$\mathrm{S5}_k(\mathcal{L})$ is finitely strongly complete with respect to $\mathrm{S5}_k([0,1]_\textrm{\L})$.
\end{theorem}

\begin{proof}
To prove the soundness implication it is enough to show that any safe structure $\mathbf{K} := \langle X, e, [0,1]_\textrm{\L}\rangle$ with $|X| \leq k$ is a model of the axiom $W_k$. Indeed, put $X = \{x_j: 1 \leq j \leq r\}$, where $r \leq k$, let $\varphi_1,\ldots,\varphi_{k+1}$ be formulas, and let $a_{ij} := \|\varphi_i\|_{\mathbf{K},x_j}$, $1 \leq i \leq k+1$, $1 \leq j \leq r$. For each $j$, let $i_j$ be such that $a_{i_jj} = \min \{a_{ij}: 1 \leq i \leq k+1\}$. Choose $i^*$ in $\{1,\ldots,k+1\} \setminus \{i_1,\ldots,i_r\}$. Then 
\[
\inf_{1 \leq i < i' \leq k+1} (a_{ij} \vee a_{i'j}) \leq a_{i^*j}
\]
for $1 \leq j \leq r$. Hence, for every $x \in X$ we have
\begin{align*}
\left\|\bigwedge_{1 \leq i < i' \leq k+1} \square(\varphi_i \vee \varphi_{i'})\right\|_{\mathbf{K},x} & = \inf_{1 \leq i < i' \leq k+1} \|\square(\varphi_i \vee \varphi_{i'})\|_{\mathbf{K},x} \\
& = \inf_{1 \leq i < i' \leq k+1} \left( \inf_{1 \leq j \leq r} \|\varphi_i \vee \varphi_{i'}\|_{\mathbf{K},x_j} \right) \\
& = \inf_{1 \leq i < i' \leq k+1} \left( \inf_{1 \leq j \leq r} (a_{ij} \vee a_{i'j}) \right) \\
& \leq \inf_{1 \leq j \leq r} a_{i^*j} \\
& = \|\square \varphi_{i^*}\|_{\mathbf{K},x} \\
& \leq \left\|\bigvee_{i=1}^{k+1} \square \varphi_i\right\|_{\mathbf{K},x}.
\end{align*}
This proves that $\mathbf{K}$ is a model of $W_k$.

Conversely, assume $\{\varphi_1,\ldots,\varphi_n\} \nvDash_{\mathrm{S5}_k(\mathcal{L})} \varphi$ for formulas  $\varphi_1, \ldots, \varphi_n, \varphi$ in the monadic language. Since $\mathbb{MMV}_k$ is the equivalent algebraic semantics of $\mathrm{S5}_k(\mathcal{L})$, we have that $\mathbb{MMV}_k \nvDash \bigwedge_{i=1}^n \varphi_i \approx 1 \to \varphi \approx 1$. Using Corollary \ref{COR: mmv de ancho k generadas por 01 a la k}, there is a valuation $h$ in $\langle [0,1]_\textrm{\L}^k, \exists_\vee, \forall_\wedge\rangle$ such that $h(\varphi_i) = 1$ for $1 \leq i \leq n$, but $h(\varphi) \ne 1$. Let $\mathbf{K} := \langle \{1,\ldots,k\}, e, [0,1]_\textrm{\L}\rangle$, where $e\colon Prop \times \omega \to [0,1]_\textrm{\L}$ is defined by $e(p,m) := h(p)(m)$ for $p \in Prop$ and $m \in \{1,\ldots,k\}$. Then $\mathbf{K}$ is a safe structure which is a model of $\{\varphi_1,\ldots,\varphi_n\}$, but not a model of $\varphi$. Thus, $\{\varphi_1,\ldots,\varphi_n\} \nvDash_{\mathrm{S5}_k([0,1]_\textrm{\L})} \varphi$.
\end{proof}

\section{Infinitary S5-modal logics based on $[0,1]_\text{\rm \L}$}

The completeness theorems in the previous section were necessarily {\em finite strong} completeness theorems because the logics $\mathrm{S5}([0,1]_\textrm{\L})$ and $\mathrm{S5}_k([0,1]_\textrm{\L})$ are infinitary, while the logics $\mathrm{S5}(\mathcal{L})$ and $\mathrm{S5}_k(\mathcal{L})$ are finitary. In this section we add an infinary rule to the latter calculi to get strong completeness theorems.

In \cite{Kulacka18} A. Kulacka adds an infinitary rule to Hájek's BL calculus and obtains a strong completeness theorem for the class of continuous t-norms. We consider the following modification of Ku\l acka's infinitary rule:
\begin{itemize}
\item[] $\square$Inf: $\displaystyle \underline{\square \varphi \vee (\square \alpha \imp (\square \beta)^n) \text{ for every } n \in \mathbb{N}} \atop \displaystyle {\square \varphi \vee (\square \alpha \imp \square \alpha * \square \beta)}$.
\end{itemize}
We define the logic $\mathrm{S5}(\mathcal{L})_\infty$ as the result of adding $\square$Inf to $\mathrm{S5}(\mathcal{L})$.

If $\Gamma \cup \{\varphi\}$ is a set of formulas of $\mathrm{S5}(\mathcal{L})_\infty$, the notation $\Gamma \vdash_{\mathrm{S5}(\mathcal{L})_\infty} \varphi$ means that there is a family of formulas $\{\psi_i: i \leq \xi\}$ indexed by a successor ordinal $\xi^+$ such that $\psi_\xi = \varphi$ and for each $i \leq \xi$ the formula $\varphi_i$ is an instance of an axiom or is the lower formula of an instance of an inference rule whose upper formulas are contained in the set $\{\psi_j: j < i\}$. The family $\{\psi_i: i \leq \xi\}$ is said to be a {\em proof of $\varphi$ from $\Gamma$ of length $\xi$}.

To avoid the cumbersome notation $\vdash_{\mathrm{S5}(\mathcal{L})_\infty}$, we simply write $\vdash$ in the rest of the section.

We aim to show that this calculus is strongly complete with respect to the logic $\mathrm{S5}([0,1]_\textrm{\L})$. We begin by proving the soundness part.

\begin{theorem} \label{TEO: strong standard soundness}
Let $\Gamma \cup \{\varphi\} \subseteq Fm$, then
\[
\Gamma \vdash \varphi  \text{ implies } \Gamma \vDash_{\mathrm{S5}([0,1]_\textrm{\rm \L})} \varphi.
\]
\end{theorem}

\begin{proof}
By Theorem \ref{TEO: finite strong completeness - width k}, it is enough to show that $\{\square \varphi \vee (\square \alpha \imp (\square \beta)^n): n \in \mathbb{N}\} \vDash_{\mathrm{S5}([0,1]_\textrm{\L})} \square \varphi \vee (\square \alpha \imp \square \alpha * \square \beta)$ for any $\alpha, \beta, \varphi \in Fm$. Let $\mathbf{K} := \langle X, e, [0,1]_\textrm{\L}\rangle$ be a model of $\square \varphi \vee (\square \alpha \imp (\square \beta)^n)$ for every $n \in \mathbb{N}$, that is, $\|\square \varphi \vee (\square \alpha \imp (\square \beta)^n)\|_{\mathbf{K},x} = 1$ for every $x \in X$ and $n \in \mathbb{N}$. Fix $x \in X$. If $\|\square \varphi\|_{\mathbf{K},x} = 1$, then $\|\square \varphi \vee (\square \alpha \imp \square \alpha * \square \beta)\|_{\mathbf{K},x} = 1$. Thus, we may assume that $\|\square \varphi\|_{\mathbf{K},x} \ne 1$. Then, $\|\square \alpha \imp (\square \beta)^n\|_{\mathbf{K},x} = 1$ for every $n \in \mathbb{N}$, so $\|\square \alpha\|_{\mathbf{K},x} \leq \|\square \beta\|_{\mathbf{K},x}^n$ for every $n \in \mathbb{N}$. Since, $[0,1]_\textrm{\L}$ is simple, we get that $\|\square \beta\|_{\mathbf{K},x} = 1$ or $\|\square \alpha\|_{\mathbf{K},x} = 0$. In any case, $\|\square \varphi \vee (\square \alpha \imp \square \alpha * \square \beta)\|_{\mathbf{K},x} = 1$.
\end{proof}

\begin{remark} \rm \label{OBS: props de la logica}
Let $\alpha$ be any formula and let $\nu$ be any BL-propositional combination of formulas that start with $\square$ or $\lozenge$. In the proofs below we make use of the following facts:
\begin{enumerate}[$(1)$]
\item \label{OBS: para todo de una o pasa a la primera} $\vdash \square(\alpha \vee \nu) \imp (\square \alpha \vee \nu)$,
\item \label{OBS: para todo de una o pasa a la segunda} $\vdash \square(\nu \vee \alpha) \imp (\nu \vee \square \alpha)$,
\item \label{OBS: para todo de una formula constante} $\vdash \nu \leftrightarrow \square \nu$.
\end{enumerate}
(1) is an axiom of $\mathrm{S5}(\mathcal{L})$, while (2) and (3) are straightforward.
\end{remark}

\begin{lemma}
Let $\Gamma \cup \{\alpha,\beta,\psi\}$ be a set of formulas. Then,
\[
\Gamma,\square\beta \vdash \psi \text{ implies } \Gamma,\square\alpha\vee\square\beta \vdash \square \alpha \vee\psi.
\]
\end{lemma}

\begin{proof}
We proceed by induction on the length of the proof of $\psi$ from $\Gamma, \square \beta$. Suppose first that there is a proof of $\psi$ from $\Gamma, \square \beta$ of length 1. If $\psi$ is an axiom or $\psi \in \Gamma$, then $\Gamma, \square \alpha \vee \square \beta \vdash \square \alpha \vee \psi$ since $\psi \vdash \square \alpha \vee \psi$. If $\psi = \square \beta$, then $\square \alpha \vee \psi = \square \alpha \vee \square \beta$ and the conclusion is trivial. Assume now that there is a proof of $\psi$ from $\Gamma, \square \beta$ of length $\xi$ and that the conclusion holds for formulas $\psi'$ whose proofs from $\Gamma, \square \beta$ are shorter than $\xi$. We can assume that $\psi$ follows by Modus Ponens, Necessitation or $\square$Inf. In the first case, there is a formula $\theta$ such that $\Gamma, \square \beta \vdash \theta$ and $\Gamma, \square \beta \vdash \theta \imp \psi$ with proofs shorter than $\xi$. By the induction hypothesis, $\Gamma, \square \alpha \vee \square \beta \vdash \square \alpha \vee \theta$ and $\Gamma, \square \alpha \vee \square \beta \vdash \square \alpha \vee (\theta \imp \psi)$. Since $\vdash (\chi \vee \theta) \imp ((\chi \vee (\theta \imp \psi)) \imp (\chi \vee \psi))$ (this holds in Basic Logic), we get that $\Gamma, \square \alpha \vee \square \beta \vdash \square \alpha \vee \psi$. If $\psi$ follows from Necessitation, then $\psi = \square \chi$ and $\Gamma, \square \beta \vdash \chi$ in less than $\xi$ steps. By the induction hypothesis, $\Gamma, \square \alpha \vee \square \beta \vdash \square \alpha \vee \chi$. Using Necessitation, Remark \ref{OBS: props de la logica}.$(\ref{OBS: para todo de una o pasa a la segunda})$ and Modus Ponens, we get that $\Gamma, \square \alpha \vee \square \beta \vdash \square \alpha \vee \square \chi$, as was to be proved. Finally, if $\psi$ follows from $\square$Inf, then $\psi = \square \chi \vee (\square \rho \imp \square \rho * \square \sigma)$ and $\Gamma, \square \beta \vdash \square \chi \vee (\square \rho \imp (\square \sigma)^n)$ in less than $\xi$ steps for every $n \in \mathbb{N}$. By the induction hypothesis, $\Gamma, \square \alpha \vee \square \beta \vdash \square \alpha \vee \square \chi \vee (\square \rho \imp (\square \sigma)^n)$ for every $n \in \mathbb{N}$. By Remark \ref{OBS: props de la logica}.$(\ref{OBS: para todo de una formula constante})$, we get that $\Gamma, \square \alpha \vee \square \beta \vdash \square(\square \alpha \vee \square \chi) \vee (\square \rho \imp (\square \sigma)^n)$ for every $n \in \mathbb{N}$. Using now $\square$Inf, we get that $\Gamma, \square \alpha \vee \square \beta \vdash \square(\square \alpha \vee \square \chi) \vee (\square \rho \imp \square \rho * \square \sigma)$. Since $\vdash (\square \gamma \vee \delta) \imp (\gamma \vee \delta)$, we finally obtain $\Gamma, \square \alpha \vee \square \beta \vdash \square \alpha \vee \square \chi \vee (\square \rho \imp \square \rho * \square \sigma)$, that is, $\Gamma, \square \alpha \vee \square \beta \vdash \square \alpha \vee \psi$.
\end{proof}

\begin{lemma} \label{LEMA: disyuncion en la hipotesis}
Let $\Gamma \cup \{\alpha,\beta,\varphi,\psi\}$ be a set of formulas. Then,
\[
\Gamma,\square\alpha \vdash \varphi \text{ and } \Gamma,\square\beta \vdash \psi \text{ imply } \Gamma,\square\alpha\vee\square\beta \vdash \varphi\vee\psi.
\]
\end{lemma}

\begin{proof}
We proceed by transfinite induction on the length of a proof of $\varphi$ from $\Gamma, \square \alpha$. Suppose first there is such a proof of length 1. If $\varphi$ is an axiom or belongs to $\Gamma$, clearly $\Gamma,\square\alpha\vee\square\beta \vdash \varphi\vee\psi$ since $\varphi \vdash \varphi\vee\psi$. If $\varphi = \square \alpha$, the conclusion follows from the previous lemma.

Suppose now that there is a proof of $\varphi$ from $\Gamma, \square \alpha$ of length $\xi$ and that the theorem holds (forall all $\beta$ and $\psi$) if this length were smaller than $\xi$. We can assume $\varphi$ follows from Modus Ponens, Necessitation or $\square$Inf. The first case follows as above and is left to the reader. In case $\varphi$ follows from Necessitation we have that $\varphi = \square \chi$ and $\Gamma, \square \alpha \vdash \chi$. Since $\Gamma, \square \beta \vdash \psi$, we have that $\Gamma, \square \beta \vdash \square \psi$. Using the induction hypothesis on $\Gamma, \square \alpha \vdash \varphi$ and $\Gamma, \square \beta \vdash \square \psi$, we get that $\Gamma, \square \alpha \vee \square \beta \vdash \chi \vee \square \psi$. By Necessitation, $\Gamma, \square \alpha \vee \square \beta \vdash \square(\chi \vee \square \psi)$ and, using the axiom $\square(\chi \vee \square \psi) \imp (\square \chi \vee \square \psi)$ and Modus Ponens, we get that $\Gamma, \square \alpha \vee \square \beta \vdash \square \chi \vee \square \psi$. Finally, from the fact that $\vdash (\gamma \vee \square\delta) \imp (\gamma \vee \delta)$, we get that $\Gamma, \square \alpha \vee \square \beta \vdash \square \chi \vee \psi$. In case $\varphi$ follows from $\square$Inf, we have that $\varphi = \square \chi \vee (\square \rho \imp \square \rho * \square \sigma)$ and that $\Gamma, \square \alpha \vdash \square \chi \vee (\square \rho \imp (\square \sigma)^n))$ in less than $\xi$ steps for every $n \in \mathbb{N}$. As in the second case, we have that $\Gamma, \square \beta \vdash \square \psi$, so the induction hypothesis implies that $\Gamma, \square \alpha \vee \square \beta \vdash (\square \chi \vee (\square \rho \imp (\square \sigma)^n)) \vee \square \psi$. Commuting the last disjunction and proceeding as in the previous paragraph, we get that $\Gamma, \square \alpha \vee \square \beta \vdash \square \psi \vee \square \chi \vee (\square \rho \imp \square \rho * \square \sigma)$. Using commutativity again and the fact that $\vdash (\gamma \vee \square \delta) \imp (\gamma \vee \delta)$, we get that $\Gamma, \square \alpha \vee \square \beta \vdash \square \chi \vee (\square \rho \imp \square \rho * \square \sigma) \vee \psi$, as was to be proved.
\end{proof}

\begin{corollary} \label{CORO: prelinealidad monadica}
Let $\Gamma \cup \{\alpha,\beta,\varphi\}$ be a set of formulas. Then,
\[
\Gamma, \square \alpha \imp \square \beta \vdash \varphi \text{ and } \Gamma, \square \beta \imp \square \alpha  \vdash \varphi \text{ imply } \Gamma \vdash \varphi.
\]
\end{corollary}

\begin{proof}
Apply the lemma noting that $\vdash \square(\square \alpha \imp \square \beta) \leftrightarrow (\square \alpha \imp \square \beta)$ to obtain that $\Gamma, (\square \alpha \imp \square \beta) \vee (\square \beta \imp \square \alpha) \vdash \varphi \vee \varphi$. Then use Basic Logic to get the final result.
\end{proof}

\begin{corollary} \label{CORO: disyuncion en la conclusion}
Let $\Gamma \cup \{\alpha,\beta,\varphi\}$ be a set of formulas. Then,
\[
\Gamma \vdash \varphi \vee \square \psi \text{ and } \Gamma, \square \psi \vdash \varphi \text{ imply } \Gamma \vdash \varphi.
\]
\end{corollary}

\begin{proof}
Since $\Gamma, \square \varphi \vdash \varphi$ and $\Gamma, \square \psi \vdash \varphi$, by Lemma \ref{LEMA: disyuncion en la hipotesis} and the fact that $\vdash (\varphi \vee \varphi) \imp \varphi$, we have that $\Gamma, \square \varphi \vee \square \psi \vdash \varphi$. By hypothesis $\Gamma \vdash \varphi \vee \square \psi$, so, using Necessitation, Remark \ref{OBS: props de la logica}.$(\ref{OBS: para todo de una o pasa a la primera})$ and Modus Ponens, we get that $\Gamma \vdash \square \varphi \vee \square \psi$. Thus, $\Gamma \vdash \varphi$.
\end{proof}

A set of formulas $\Gamma$ is said to be a {\em theory of ${\mathrm{S5}(\mathcal{L})_\infty}$} provided that for every formula $\varphi$ such that $\Gamma \vdash \varphi$ we have that $\varphi \in \Gamma$. We say that a theory $\Gamma$ is {\em $\square$-prelinear} if for every formulas $\alpha, \beta$ we have that $\square \alpha \imp \square \beta \in \Gamma$ or $\square \beta \imp \square \alpha \in \Gamma$. The following theorem is an adaptation of \cite[Theorem 14]{Kulacka18} to the monadic case.

\begin{theorem} \label{TEO: prelinear extension}
Let $\Gamma \cup \{\varphi\}$ be a set of formulas such that $\Gamma \nvdash \varphi$. Then, there is a $\square$-prelinear theory $\Gamma^*$ containing $\Gamma$ such that $\Gamma^* \nvdash \varphi$. 
\end{theorem}

\begin{proof}
We define a sequence of sets $\Gamma_0, \Gamma_1, \ldots$ and a sequence of formulas $\varphi_0, \varphi_1, \ldots$ such that $\Gamma_n \nvdash \varphi_n$ for every $n \geq 0$. Let $\Gamma_0 := \Gamma$ and $\varphi_0 := \varphi$. Let $\alpha_0, \alpha_1, \ldots$ be an enumeration of all the formulas. Assume $\Gamma_n$ and $\varphi_n$ have already been defined so that $\Gamma_n \nvdash \varphi_n$. We define $\Gamma_{n+1}$ and $\varphi_{n+1}$ according to the following rules:
\begin{itemize}
\item If $\Gamma_n, \alpha_n \nvdash \varphi_n$, then $\Gamma_{n+1} := \Gamma_n \cup \{\alpha_n\}$ and $\varphi_{n+1} := \varphi_n$. Thus, $\Gamma_{n+1} \nvdash \varphi_{n+1}$.
\item If $\Gamma_n, \alpha_n \vdash \varphi_n$, then $\Gamma_{n+1} := \Gamma_n$ and
\begin{itemize}
\item if $\alpha_n$ is not of the form $\square \psi \vee (\square \alpha \imp \square \alpha * \square \beta)$, then $\varphi_{n+1} := \varphi_n$. Thus, $\Gamma_{n+1} \nvdash \varphi_{n+1}$.
\item if $\alpha_n$ is of the form $\square \psi \vee (\square \alpha \imp \square \alpha * \square \beta)$, then $\varphi_{n+1} := \varphi_n \vee \square \psi \vee (\square \alpha \imp (\square \beta)^k)$ where $k$ is the smallest natural number such that $\Gamma_n \nvdash \varphi_n \vee \square \psi \vee (\square \alpha \imp (\square \beta)^k)$. Such $k$ exists because, if $\Gamma_n \vdash \varphi_n \vee \square \psi \vee (\square \alpha \imp (\square \beta)^m)$ for every $m \in \mathbb{N}$, then, using Necessitation, Remark \ref{OBS: props de la logica}.$(\ref{OBS: para todo de una o pasa a la primera})$ and Modus Ponens, we get that $\Gamma_n \vdash \square(\varphi_n \vee \square \psi) \vee (\square \alpha \imp (\square \beta)^m)$ for every $m \in \mathbb{N}$; thus, by $\square$Inf, we have that $\Gamma_n \vdash \square(\varphi_n \vee \square \psi) \vee (\square \alpha \imp \square \alpha * \square \beta)$. Since $\vdash (\square \gamma \vee \delta) \imp (\gamma \vee \delta)$, we get that $\Gamma_n \vdash \varphi_n \vee \square \psi \vee (\square \alpha \imp \square \alpha * \square \beta)$, that is, $\Gamma_n \vdash \varphi_n \vee \alpha_n$. Note that $\alpha_n$ is a BL-propositional combination of formulas that start with $\square$; hence, by Remark \ref{OBS: props de la logica}.$(\ref{OBS: para todo de una formula constante})$, we have that $\vdash \alpha_n \leftrightarrow \square \alpha_n$. Using Basic Logic properties we can then derive that $\Gamma_n \vdash \varphi_n \vee \square \alpha_n$. Moreover, since $\Gamma_n, \alpha_n \vdash \varphi_n$, we also have that $\Gamma_n, \square \alpha_n \vdash \varphi_n$. Using Corollary \ref{CORO: disyuncion en la conclusion}, we obtain $\Gamma_n \vdash \varphi_n$, a contradiction. Thus, $\Gamma_{n+1} \nvdash \varphi_{n+1}$.
\end{itemize}
\end{itemize}
Observe that, by construction, for $n \leq m$ we have that $\Gamma_n \subseteq \Gamma_m$ and $\vdash \varphi_n \imp \varphi_m$. Clearly, $\Gamma_n \nvdash \varphi_n$ for every $n \geq 0$. Moreover, for every $n,m \geq 0$ we have that $\Gamma_n \nvdash \varphi_m$. Indeed, suppose $\Gamma_n \vdash \varphi_m$. If $n \leq m$, then $\Gamma_n \subseteq \Gamma_m$ and $\Gamma_m \vdash \varphi_m$, a contradiction. If $m \leq n$, then $\vdash \varphi_m \imp \varphi_n$ and $\Gamma_n \vdash \varphi_n$, again a contradiction.

Put $\Gamma^{*} := \bigcup \Gamma_n$. Let us show that $\Gamma^*$ is a theory of $\mathrm{S5}(\mathcal{L})_\infty$. It is enough to show that $\Gamma^*$ contains all the instances of axioms and is closed under the inference rules. Let $\alpha$ be an instance of an axiom of $\mathrm{S5}(\mathcal{L})_\infty$. Then, there is $n \geq 0$ such that $\alpha_n = \alpha$. Being $\alpha_n$ an instance of an axiom, if $\Gamma_n,\alpha_n \vdash \varphi_n$, then $\Gamma_n \vdash \varphi_n$, a contradiction. Thus, $\Gamma_n, \alpha_n \nvdash \varphi_n$ whence, by construction, $\alpha_n \in \Gamma_{n+1} \subseteq \Gamma^*$. Assume now that $\alpha, \alpha \imp \beta \in \Gamma^*$. Let $n$ be such that $\beta = \alpha_n$. If $\Gamma_n, \alpha_n \nvdash \varphi_n$, then $\beta = \alpha_n \in \Gamma_{n+1} \subseteq \Gamma^*$ and we are done. Otherwise, $\Gamma_n, \beta \vdash \varphi_n$. Let $m$ be sufficiently large such that $n \leq m$ and $\alpha, \alpha \imp \beta \in \Gamma_m$. Then $\Gamma_m \vdash \beta$ and $\Gamma_m, \beta \vdash \varphi_n$, so $\Gamma_m \vdash \varphi_n$, a contradiction. This shows that $\Gamma^*$ is closed under Modus Ponens. We now show that $\Gamma^*$ is closed under Necessitation. Suppose that $\alpha \in \Gamma^*$ and let $n$ be such that $\alpha_n = \square \alpha$. If $\Gamma_n, \alpha_n \nvdash \varphi_n$, then $\square \alpha = \alpha_n \in \Gamma_{n+1} \subseteq \Gamma^*$ and we are done. Otherwise, $\Gamma_n, \square \alpha \vdash \varphi_n$. Let $m$ be sufficiently large such that $m \geq n$ and $\alpha \in \Gamma_m$. Then, $\Gamma_m \vdash \alpha$ and $\Gamma_m, \square \alpha \vdash \varphi_n$. From $\Gamma_m \vdash \alpha$ and Necessitation, we get that $\Gamma_m \vdash \square \alpha$, so $\Gamma_m \vdash \varphi_n$, a contradiction. Finally, let us prove that $\Gamma^*$ is closed under $\square$Inf. Suppose $\square \psi \vee (\square \alpha \imp (\square \beta)^m)) \in \Gamma^*$ for every $m \in \mathbb{N}$. Let $n$ be such that $\alpha_n = \square \psi \vee (\square \alpha \imp \square \alpha * \square \beta)$. If $\Gamma_n, \alpha_n \nvdash \varphi_n$, then $\alpha_n \in \Gamma_{n+1} \subseteq \Gamma^*$ and we are done. Otherwise, $\Gamma_n, \alpha_n \vdash \varphi_n$. By construction, there is $k \in \mathbb{N}$ such that $\varphi_{n+1} = \varphi_n \vee \square \psi \vee (\square \alpha \imp (\square \beta)^k)$. Since $\square \psi \vee (\square \alpha \imp (\square \beta)^k) \in \Gamma^*$, there is a sufficiently large $m$ such that $m \geq n$ and $\square \psi \vee (\square \alpha \imp (\square \beta)^k) \in \Gamma_m$. Thus, since $\vdash \delta \imp (\gamma \vee \delta)$, we have that $\Gamma_m \vdash \varphi_{n+1}$, a contradiction.

It remains to show that $\Gamma^*$ is $\square$-prelinear. Let $\alpha, \beta$ be formulas and let $n,m$ be such that $\alpha_n = \square \alpha \imp \square \beta$ and $\alpha_m = \square \beta \imp \square \alpha$. If $\Gamma_n, \alpha_n \nvdash \varphi_n$, then $\alpha_n \in \Gamma_{n+1} \subseteq \Gamma^*$ and we are done. Analogously, in case $\Gamma_m, \alpha_m \nvdash \varphi_m$ we have that $\alpha_m \in \Gamma^*$. Thus, we can assume that $\Gamma_n, \alpha_n \vdash \varphi_n$ and $\Gamma_m, \alpha_m \vdash \varphi_m$. Letting $k := \max\{m,n\}$, we have that $\Gamma_k, \alpha_n \vdash \varphi_k$ and $\Gamma_k, \alpha_m \vdash \varphi_k$. By Corollary \ref{CORO: prelinealidad monadica}, $\Gamma_k \vdash \varphi_k$, a contradiction.
\end{proof}

Observe that $\mathrm{S5}(\mathcal{L})_\infty$ is an implicative logic, since it is an extension of $\mathrm{S5}(\mathcal{L})$, which is implicative (see \cite{CCDVR17}). Given a set $\Gamma \subseteq Fm$, let $\equiv_\Gamma$ be the relation on $Fm$ given by $\varphi \equiv_\Gamma \psi$ iff $\Gamma \vdash \varphi \imp \psi$ and $\Gamma \vdash \psi \imp \varphi$. Then, $\equiv_\Gamma$ is an equivalence relation on $Fm$. We write $[\varphi]_\Gamma$ for the equivalence class of $\varphi$ modulo $\equiv_\Gamma$. We can define an algebra $\mathbf{L}_\Gamma := \langle Fm/\mathord{\equiv}_\gamma, \wedge, \vee, *, \imp, 0, 1, \exists, \forall\rangle$ where
\begin{itemize}
\item $0 := [\bar{0}]_\Gamma$, $1 := [\bar{1}]_\Gamma$,
\item $[\varphi]_\Gamma \star [\psi]_\Gamma := [\varphi \star \psi]_\Gamma$ for $\star \in \{\wedge, \vee, *, \imp\}$,
\item $\forall [\varphi]_\Gamma := [\square \varphi]_\Gamma$, $\exists [\varphi]_\Gamma := [\lozenge \varphi]_\Gamma$.
\end{itemize}
Then, $\mathbf{L}_\Gamma$ is a monadic MV-algebra.

\begin{lemma} \label{LEMA: MV-chain + arquim = simple}
Let $\mathbf{A}$ be an MV-chain with the following property: for $a,b \in A$, if $a \leq b^n$ for every $n \geq 1$, then $a = a*b$. Then, $\mathbf{A}$ is simple.
\end{lemma}

\begin{proof}
Fix $b \in A$ with $b \ne 1$. If $a$ is a lower bound of the set $\{b^n: n \in \mathbb{N}\}$, then, by hypothesis, $a = a*b$. Thus $\neg a = \neg a \oplus \neg b = b \imp \neg a$. Hence $b \vee \neg a = (b \imp \neg a) \imp \neg a = \neg a \imp \neg a = 1$. Since $\mathbf{A}$ is a chain and $b \ne 1$, we get that $\neg a = 1$, that is, $a = 0$. This proves that $\inf\{b^n : n \in \mathbb{N}\} = 0$. Now, since $b \ne 1$, we have that $\neg b \ne 0$, thus there must exist $n \in \mathbb{N}$ such that $b^n \leq \neg b$, so $b^{n+1} \leq b * \neg b = 0$.
\end{proof}

\begin{lemma} \label{LEMA: el alg de Lindenbaum es monadica simple}
If $\Gamma$ is $\square$-prelinear, then $\mathbf{L}_\Gamma$ is a simple monadic MV-algebra.
\end{lemma}

\begin{proof}
Clearly, $\square$-prelinarity implies that $\exists \mathbf{L}_\Gamma$ is totally ordered. However, the rule $\square$Inf implies a stronger condition on $\mathbf{L}_\Gamma$. Note that taking $\varphi = \bar{0}$ in $\square$Inf we have that $\Gamma \vdash \square \alpha \imp \square \alpha * \square \beta$ whenever $\Gamma \vdash \square \alpha \imp (\square \beta)^n $ for every $n \in \mathbb{N}$. This translates to the following property of $\mathbf{L}_\Gamma$: for $a,b \in L_\Gamma$, if $\forall a \leq (\forall b)^n$ for every $n \in \mathbb{N}$, then $\forall a = \forall a * \forall b$. Lemma \ref{LEMA: MV-chain + arquim = simple} implies that $\exists \mathbf{L}_\Gamma$ is a simple MV-algebra and, thus, $\mathbf{L}_\Gamma$ is a simple monadic MV-algebra.
\end{proof}

Simple monadic MV-algebras have a nice functional representation that we give in the following theorem.

\begin{theorem} \label{TEO: embedding de monadicas simples en funcionales estandar}
Let $\langle \mathbf{A}, \exists, \forall\rangle$ be a simple monadic MV-algebra. Then, there is a set $I$ and an embedding of $\langle \mathbf{A}, \exists, \forall\rangle$ into $\langle [0,1]_\textrm{\rm \L}^I, \exists_\vee, \forall_\wedge\rangle$. Moreover, the set $I$ can be taken to be the set of maximal filters of $\mathbf{A}$.
\end{theorem}

\begin{proof}
Let $I$ be the set of maximal filters of $\mathbf{A}$. The filter $R := \bigcap \{M: M \in I\}$ is known as the radical of $\mathbf{A}$. It is well-known that $R = \{a \in A: 2a^n = 1 \text{ for every } n \in \mathbb{N}\}$. Suppose there is $a \in R$ with $a \ne 1$. Then, $2a^n = 1$ for every $n \in \mathbb{N}$. Using Lemma \ref{LEMA: prop de los cuantif en MMV}, we have that $1 = \forall 1 = \forall (2a^n) = 2(\forall a)^n$ for every $n \in \mathbb{N}$. Thus, $\forall a \in R$. Since $\forall a \leq a < 1$ and $\exists \mathbf{A}$ is a simple MV-algebra, there is $m \in \mathbb{N}$ such that $(\forall a)^m = 0$. Hence $1 = 2(\forall a)^m = 0$, a contradiction. This shows that $R = \{1\}$.

For each $M \in I$, let $g_M\colon \mathbf{A}/M \to [0,1]_\textrm{\L}$ be the unique embedding of the simple MV-algebra $\mathbf{A}/M$ into $[0,1]_\textrm{\L}$ (see \cite[Theorem 3.5.1 and Corollary 7.2.6]{CDM00}). Moreover, since $M \cap \exists A$ is a proper filter of the simple MV-algebra $\exists \mathbf{A}$, we have that $M \cap \exists A = \{1\}$. Letting $q_M$ be the restriction to $\exists A$ of the canonical embedding, we have that $g_M \circ q_M\colon \exists \mathbf{A} \to [0,1]_\textrm{\L}$ is an embedding. Since there is a unique embedding of $\exists \mathbf{A}$ into $[0,1]_\textrm{\L}$, we must have that $(g_M \circ q_M)(\exists a) = (g_{M'} \circ q_{M'})(\exists a)$, that is, $g_M((\exists a)/M) = g_{M'}((\exists a)/M')$ for every $a \in A$ and $M,M' \in I$, 

We define $f\colon \mathbf{A} \to [0,1]_\textrm{\L}^I$ by $f(a)(M) = g_M(a/M)$ for every $M \in I$. As $R = \{1\}$ the map $f$ is an embedding. In order to prove that $f\colon \langle \mathbf{A}, \exists, \forall\rangle \to \langle [0,1]_\textrm{\L}^I, \exists_\vee, \forall_\wedge\rangle$ is an embedding, it remains to show that $f(\exists a)(M) = \sup \{f(a)(M'): M' \in I\}$ and $f(\forall a)(M) = \inf \{f(a)(M'): M' \in I\}$ for every $a \in A$ and $M \in I$. By Lemma \ref{LEMA: prop de los cuantif en MMV}, we know that $\forall x = \neg \exists \neg x$ in every monadic MV-algebra, so it is enough to show the first of the previous equalities. Fix $a \in A$. Note that $(\exists a \imp a)^n \ne 0$ for every $n \in \mathbb{N}$. Indeed, if $(\exists a \imp a)^n = 0$, then, using Lemma \ref{LEMA: prop de los cuantif en MMV}, we get that $0 = \exists 0 = \exists (\exists a \imp a)^n = (\exists (\exists a \imp a))^n = 1$, a contradiction. Thus, there must exist $M^* \in I$ such that $\exists a \imp a \in M^*$. Since $a \leq \exists a$, we have that $a \to \exists a = 1 \in M^*$ as well; so $a/M^* = (\exists a)/M^*$. Thus, $f(a)(M^*) = g_{M^*}(a/M^*) = g_{M^*}((\exists a)/M^*) = f(\exists a)(M^*)$. In addition, for every $M' \in I$ we have that $a/M' \leq (\exists a)/M'$, so $f(a)(M') = g_{M'}(a/M') \leq g_{M'}((\exists a)/M') = g_{M^*}((\exists a)/M^*) = f(\exists a)(M^*)$. This shows that $\sup \{f(a)(M'): M' \in I\} = f(\exists a)(M^*) = f(\exists a)(M)$ for every $M \in I$.
\end{proof}

We are finally ready to prove the standard completeness theorem.

\begin{theorem} \label{TEO: completitud estandar infinitaria}
Let $\Gamma \cup \{\varphi\} \subseteq Fm$, then
\[
\Gamma \vdash \varphi  \text{ if and only if } \Gamma \vDash_{\mathrm{S5}([0,1]_\textrm{\rm \L})} \varphi.
\]
\end{theorem}

\begin{proof}
The soundness implication was already proved in \ref{TEO: strong standard soundness}. Suppose $\Gamma \nvdash \varphi$. By Theorem \ref{TEO: prelinear extension}, there is a $\square$-prealinear theory $\Gamma^*$ containing $\Gamma$ such that $\Gamma^* \nvdash \varphi$, and, by Lemma \ref{LEMA: el alg de Lindenbaum es monadica simple}, the $\mathbf{L}_{\Gamma^*}$ is a simple monadic MV-algebra. Using now Lemma {}, let $h\colon \mathbf{L}_{\Gamma^*} \to \langle [0,1]_\textrm{\L}^I, \exists_\vee, \forall_\wedge\rangle$ be an embedding for some nonempty set $I$. Define $\mathbf{K} := \langle I, e, [0,1]_\textrm{\L}\rangle$ where $e\colon I \times Prop \to [0,1]$ is defined by $e(i,p) := h([p]_{\Gamma^*})(i)$ for every $i \in I$ and $p \in Prop$. By induction on the structure of formulas it follows that $\|\psi\|_{\mathbf{K},i} = h([\psi]_{\Gamma^*})(i)$ for every $i \in I$ and $\psi \in Fm$. Thus, $\mathbf{K}$ is a safe model of $\Gamma^*$, and also of $\Gamma$. However, since $\Gamma^* \nvdash \varphi$, we have that $[\varphi]_{\Gamma^*} \ne 1$, so there is $i \in I$ such that $\|\varphi\|_{\mathbf{K},i} = h([\varphi]_{\Gamma^*})(i) \ne 1$. This shows that $\Gamma \nvDash_{\mathrm{S5}([0,1]_\textrm{\rm \L})} \varphi$.
\end{proof}

\subsection*{Alternative axiomatization of $\mathrm{S5}(\mathcal{L})_\infty$}

Consider the rule (this is the rule introduced by A. Ku\l acka in \cite{Kulacka18})
\begin{itemize}
\item[] Inf: $\displaystyle \underline{\varphi \vee (\alpha \imp \beta^n) \text{ for every } n \in \mathbb{N}} \atop \displaystyle {\varphi \vee (\alpha \imp \alpha * \beta)}$
\end{itemize}
Let $\mathrm{S5}(\mathcal{L})_\infty'$ be the extension of $\mathrm{S5}(\mathcal{L})$ by Inf.

\begin{theorem}
Let $\Gamma \cup \{\varphi\} \subseteq Fm$, then
\[
\Gamma \vdash_{\mathrm{S5}(\mathcal{L})_\infty} \varphi  \text{ if and only if } \Gamma \vdash_{\mathrm{S5}(\mathcal{L})_\infty'} \varphi.
\]
\end{theorem}

\begin{proof}
Since every instance of $\square$Inf is an instance of Inf, the forward implication is trivial. For the converse, suppose $\Gamma \vdash_{\mathrm{S5}(\mathcal{L})_\infty'} \varphi$. Exactly as in the proof of Theorem \ref{TEO: strong standard soundness}, we get that $\Gamma \vDash_{\mathrm{S5}([0,1]_\textrm{\L})} \varphi$. Finally, by Theorem \ref{TEO: completitud estandar infinitaria} we get that $\Gamma \vdash_{\mathrm{S5}(\mathcal{L})_\infty} \varphi$.
\end{proof}

\subsection*{The bounded universe case}

Given a natural number $k$, let $\mathrm{S5}_k(\mathcal{L})_\infty$ be the axiomatic extension of $\mathrm{S5}(\mathcal{L})_\infty$ by the axiom $(W_k)$. As an extension of the previous results we claim that $\mathrm{S5}_k(\mathcal{L})_\infty$ is strongly complete with respect to $\mathrm{S5}_k([0,1]_\textrm{\L})$. We start with the following basic result.

\begin{lemma}
Let $\langle \mathbf{A}, \exists, \forall\rangle$ be a FSI monadic MV-algebra of width $\leq k$. Then, $\mathbf{A}$ has at most $k$ maximal filters.
\end{lemma}

\begin{proof}
From Theorem \ref{TEO: equivalencias para ancho k}, $\mathbf{A}$ is a subdirect product of at most $k$ MV-chains. Therefore, the lemma follows.
\end{proof}

As a consequence of this lemma and Theorem \ref{TEO: embedding de monadicas simples en funcionales estandar} we get the following result.

\begin{corollary} \label{CORO: embedding de simples de ancho k en funcionales estandar de ancho k}
Let $\langle \mathbf{A}, \exists, \forall\rangle$ be a simple monadic MV-algebra of width $\leq k$. Then, there is an embedding from $\langle \mathbf{A}, \exists, \forall\rangle$ into $\langle [0,1]_\textrm{\rm \L}^r, \exists_\vee, \forall_\wedge\rangle$ for some natural number $r \leq k$.
\end{corollary}

We are now ready to prove the completeness result for the bounded universe case.

\begin{theorem} \label{TEO: strong standard completenes - width k}
Let $\Gamma \cup \{\varphi\} \subseteq Fm$, then
\[
\Gamma \vdash_{\mathrm{S5}_k(\mathcal{L})_\infty} \varphi  \text{ if and only if } \Gamma \vDash_{\mathrm{S5}_k([0,1]_\textrm{\rm \L})} \varphi.
\]
\end{theorem}

\begin{proof}
The soundness implication follows as in the proofs of Theorems \ref{TEO: finite strong completeness - width k} and \ref{TEO: strong standard soundness}. For the completeness implication, suppose $\Gamma \nvdash_{\mathrm{S5}_k(\mathcal{L})_\infty} \varphi$. Now, observe that Theorem \ref{TEO: prelinear extension} and the lemmas needed in its proof are also valid for the axiomatic extension $\mathrm{S5}_k(\mathcal{L})_\infty$. Thus, there is a $\square$-prelinear theory $\Gamma^*$ of $\mathrm{S5}_k(\mathcal{L}_\infty)$ extending $\Gamma$ such that $\Gamma^* \nvdash_{\mathrm{S5}_k(\mathcal{L}_\infty)} \varphi$. Hence, by Lemma \ref{LEMA: el alg de Lindenbaum es monadica simple} and the fact that $\Gamma^*$ is a theory of $\mathrm{S5}_k(\mathcal{L}_\infty)$, we conclude that $\mathbf{L}_{\Gamma^*}$ is a simple monadic MV-algebra satisfying the equation $(W_k)$, that is, $\mathbf{L}_{\Gamma^*}$ has width $\leq k$. Thus, by Corollary \ref{CORO: embedding de simples de ancho k en funcionales estandar de ancho k}, there is an embedding $h\colon \mathbf{L}_{\Gamma^*} \to \langle [0,1]_\textrm{\rm L}^r, \exists_\vee, \forall_\wedge\rangle$ for some natural number $r \leq k$. Finally, we proceed as in the proof of Theorem \ref{TEO: completitud estandar infinitaria} and define a model $\mathbf{K}$ of $\Gamma^*$ whose universe has $r$ elements that is not a model of $\varphi$.
\end{proof}

\section*{Acknowledgements}
We would like to thank the support of CONICET (PIP 11220170100195CO) and Departamento de Matemática (UNS) (PGI 24/L108).

\bibliographystyle{plain}

\

\

\noindent \textsc{D. Castaño} \\
Departamento de Matemática (Universidad Nacional del Sur) \\
Insituto de Matemática (INMABB) - UNS-CONICET \\
Bahía Blanca, Argentina \\
diego.castano@uns.edu.ar

\

\noindent \textsc{J. P. Díaz Varela} \\
Departamento de Matemática (Universidad Nacional del Sur) \\
Insituto de Matemática (INMABB) - UNS-CONICET \\
Bahía Blanca, Argentina \\
usdiavar@criba.edu.ar

\

\noindent \textsc{G. Savoy} \\
Departamento de Matemática (Universidad Nacional del Sur) \\
Insituto de Matemática (INMABB) - UNS-CONICET \\
Bahía Blanca, Argentina \\
gabriel\_savoy@hotmail.com

\end{document}